\documentclass[9pt,a4paper]{amsart}

\usepackage{amssymb,mathabx}
\usepackage[numbers]{natbib}
\usepackage{color}
\usepackage{graphicx}
\usepackage[pdftex,plainpages=false,colorlinks,hyperindex,bookmarksopen,linkcolor=red,citecolor=blue,urlcolor=blue]{hyperref}

\bibpunct{[}{]}{;}{n}{,}{,}
\usepackage[hmargin=3cm,vmargin={3.5cm,4cm}]{geometry}

\newtheorem{te}{Theorem}[section]

\newtheorem{os}[te]{Remark}

\newtheorem{prop}[te]{Proposition}

\numberwithin{equation}{section}

\allowdisplaybreaks

\begin{document}

    \title{State-dependent Fractional Point Processes}

    \author{Roberto Garra$^1$}
    \address{${}^1$Dipartimento di Scienze di Base e Applicate per l'Ingegneria, Sapienza Universit\`a di Roma.}

    \author{Enzo Orsingher$^2$}
    \address{${}^2$Dipartimento di Statistica, Sapienza Universit\`a di Roma.}

    \author{Federico Polito$^3$}
    \address{${}^3$Dipartimento di Matematica "G. Peano'', Universit\`a degli Studi di Torino.}


    \date{\today}

    \begin{abstract}
        The aim of this paper is the analysis of the fractional Poisson process where the state probabilities
        $p_k^{\nu_k}(t)$, $t\ge 0$, are
        governed by time-fractional equations of order $0<\nu_k\leq 1$ depending on the number $k$ of events
        occurred up to time $t$.
        We are able to obtain explicitely the Laplace transform of $p_k^{\nu_k}(t)$ and various representations
        of state probabilities.
        We show that the Poisson process with intermediate waiting times depending on $\nu_k$ differs from
        that constructed from the
        fractional state equations (in the case $\nu_k = \nu$, for all $k$, they coincide with the time-fractional
        Poisson process).
        We also introduce a different form of fractional state-dependent Poisson process as a weighted sum of
        homogeneous Poisson processes.
        Finally we consider the fractional birth process governed by equations with state-dependent fractionality.
    \end{abstract}

    \maketitle

    \section{Introduction}

        We first consider a state-dependent time-fractional Poisson process $N(t)$, $t\geq 0$, whose state
        probabilities $p_k^{\nu_k}(t)=\Pr\{N(t)=k\}$
        are governed by the following equations
        \begin{align}
            \label{aao}
            \begin{cases}
                \frac{\mathrm{d}^{\nu_k}}{\mathrm{d}t^{\nu_k}} p_k^{\nu_k}(t)
                = -\lambda p_k^{\nu_k}(t) +\lambda p_{k-1}^{\nu_{k-1}}(t), & k \ge 0, \: t > 0,
                \: \nu_k \in (0,1], \: \lambda > 0, \\
                p_k^{\nu_k}(0) =
                \begin{cases}
                    1, \quad k=0, \\
                    0, \quad k \geq 1,
                \end{cases}
            \end{cases}
        \end{align}
        where $p_k^{\nu_k}(t) = 0$, if $k \in \mathbb{Z}^- \backslash \{0\}$. These equations are obtained
        by replacing, in the governing equations of the homogeneous Poisson process, the ordinary
        derivative with the Dzhrbashyan--Caputo fractional derivative
        that is \citep{Pod}
        \begin{equation}\label{Capu}
            \frac{\mathrm d^{\nu}}{\mathrm dt^{\nu}}f(t)=
            \begin{cases}
                \frac{1}{\Gamma(m-\nu)}\int_0^{t}(t-s)^{m-\nu-1}f^{(m)} (s) \, \mathrm ds, & m-1<\nu<m , \\
                \frac{d^m f}{dt^m}, & \nu = m.
            \end{cases}
        \end{equation}
        We remark that in \eqref{aao}, the order of the fractional derivatives depend on the number of events occurred up
        to time $t$. By definition we have that
        \begin{equation}
            \frac{\mathrm d^{\nu_k}}{\mathrm dt^{\nu_k}}p_k^{\nu_k}(t)=
            \frac{1}{\Gamma(1-\nu_k)}\int_0^{t}(t-s)^{-\nu_k}\frac{\mathrm d}{\mathrm ds}p_k^{\nu_k} (s)
            \, \mathrm ds, \qquad 0<\nu_k<1.
        \end{equation}
        Hence the dependence of $p_k^{\nu_k}(t)$ on the
        past is twofold. On one side, the fractional derivative
        depends on the whole time span $[0,t]$ through the weight
        function. On the other side the number of events occurred up
        to the time $t$ modifies the power of the weight function.
        This means that the memory effect can play an increasing or
        decreasing role, in the case of a monotonical structure of
        the sequence of fractional orders $\nu_k$.

        For example, if $\nu_k$ decreases with $k$, the
        memory function tends to be constant and to give the
        same weight to the whole time span $[0,t]$. We notice that
        state-depending fractionality was considered in different
        contexts by \citet{Fedotov}.

        For $\nu_k= \nu$, for all $k$, the system \eqref{aao} coincides with the one
        governing the classical fractional Poisson process considered for example by
        \citet{Beghin}, where the fractional derivative is
        meant in the Dzhrbashyan--Caputo sense as in this case. Of course, if $\nu_k
        = 1$, for all $k$, we retrieve the governing equation for the
        homogeneous Poisson process.
        Some papers devoted to various forms of fractional Poisson
        processes have appeared in the last decades.
        In \citet{hil} the authors introduced for the first time the Mittag-Leffler waiting-time
        density in the theory of continuous-time random walks. The time-fractional Poisson process was
    then explicitly considered by \citet{rep}. Starting from this paper, different approaches to
    fractional Poisson processes were considered. In \citet{viet}, for example, the authors considered
        renewal processes with Mittag-Leffler distributed intertimes.
    A slightly different approach to the fractional Poisson process
        was developed in \citet{Laskin}, where the fractional derivative
        appearing in the equations governing the state probabilities coincides with the Riemann--Liouville derivative.
    More recently \citet{Beghin} and \citet{vel} studied the subordination of the Poisson process to the inverse
    stable subordinator, discussing the relation with fractional Poisson processes.
    Another type of fractional Poisson process was developed in
    \citet{fed} where a space-fractionality is considered.
        Physical applications of the fractional
    Poisson processes are discussed, for example,
        in \citet{Laskin1}, where a new family of quantum coherent states has been studied.

        By solving equation \eqref{aao}, we obtain that
        \begin{equation}
            \label{intr}
            \int_0^{+\infty}e^{-st}p_k^{\nu_k}(t) \mathrm dt
            = \frac{\lambda^k s^{\nu_0-1}}{\prod_{j=0}^k
            (s^{\nu_j}+\lambda)}, \qquad s > 0.
        \end{equation}
        The inversion of \eqref{intr} is by no means a simple matter and we have been able to obtain an explicit result
        for $p_0^{\nu_0}$ and $p_1^{\nu_1}$ in terms of generalized Mittag--Leffler functions defined as
        (see for example \citet{Saxena})
        \begin{equation}
            E_{\nu,\beta}^m(x)= \sum_{k=0}^{\infty}\frac{x^k\Gamma(m+k)}{k!\Gamma(\nu
            k+\beta)\Gamma(m)}, \qquad \nu, \beta, m \in\mathbb{R}^+, \: x\in \mathbb{R}.
        \end{equation}
        We give also the distribution $p_k^{\nu_k}(t)$ of the Poisson process with fractionality $\nu_k$ depending on the
        number of events $k$, in terms of subordinators and their inverses (see formula \eqref{casai1}
        below).

        A part of our paper is devoted to the construction of a
        point process $\mathcal{N}(t)$, $t\geq 0$, with intertime $U_k$
        between the $k$th and $(k+1)$th event
        distributed  as
        \begin{equation}
            \Pr\{U_k> t\}= E_{\nu_k,1}(-\lambda t^{\nu_k}).
        \end{equation}
        The Laplace transform of the univariate distributions of
        $\mathcal{N}(t)$, $t\geq 0$, is
        \begin{equation}
            \int_0^\infty e^{-s t} \Pr \{ \mathcal{N}(t) = k \} \mathrm{d}t= \lambda^k
            \frac{s^{\nu_k-1}}{\prod_{j=0}^k (s^{\nu_j}+\lambda)},
        \end{equation}
        which slightly differs from \eqref{intr}. From this point of view the state-dependent
        fractional Poisson process differs from the time-fractional Poisson process because
        the approach based on the construction by means of independent inter-event times $U_k$ and
        the one based on fractional equations \eqref{aao}, do not lead to the same one-dimensional
        distribution.
        We show that the probabilities
        $p_k(t)=\Pr \{ \mathcal{N}(t) = k \}$ are
        solutions to the fractional integral equations
        \begin{equation}
            p_k(t)-p_k(0)= -\lambda I^{\nu_k} p_k(t)+ \lambda I^{\nu_{k-1}}
            p_{k-1}(t),
        \end{equation}
        where $I^{\nu_k}$ is the Riemann--Liouville fractional integral
        \begin{equation}\label{riml}
            \left(I^{\nu_k}f\right)(t) = \frac{1}{\Gamma(\nu_k)}\int_0^{t}(t-s)^{\nu_k-1}f(s)
            \mathrm ds, \qquad \nu_k > 0.
        \end{equation}
        A third definition of the state-dependent fractional Poisson process, say $\widehat{N}(t)$, with distribution
        \begin{equation}
            \label{palombaro}
            \Pr\{\widehat{N}(t)= j\}=\frac{\frac{(\lambda t)^j}{\Gamma(\nu_j j+1)}
            \frac{1}{E_{\nu_j,1}(\lambda t)}}{\sum_{j=0}^{+\infty}
            \frac{(\lambda t)^j}{\Gamma(\nu_j j+1)}    \frac{1}{E_{\nu_j,1}(\lambda t)}}, \quad j \geq 0,
        \end{equation}
        is introduced and analyzed in Section \ref{terzo}.
        The distribution
        \begin{equation}
            \label{palomar}
            \Pr\{\widehat{N}_{\nu}(t)= j\}=\frac{(\lambda t)^k}{\Gamma(\nu j+1)}\frac{1}{E_{\nu_j,1}(\lambda t)},
        \end{equation}
        investigated in \citet{Beghin}, has been proved to be a weighted sum of Poisson distributions
        in \citet{Bala} and \citet{Macci}.

        Finally, we analyze the state-dependent nonlinear pure birth process with one initial progenitor,
        where the state probabilities
        $p_k^{\nu_k}(t)$ satisfy the fractional equations
        \begin{align}\label{intr1}
            \begin{cases}
                \frac{\mathrm{d}^{\nu_k}}{\mathrm{d}t^{\nu_k}} p_k^{\nu_k}(t)
                = -\lambda_k p_k^{\nu_k}(t) +\lambda_{k-1} p_{k-1}^{\nu_{k-1}}(t),
                & k \ge 1, \: t > 0, \: \nu_k \in (0,1], \\
                p_k^{\nu_k}(0) =
                \begin{cases}
                    1, \quad k = 1, \\
                    0, \quad k \geq 2.
                \end{cases}
            \end{cases}
        \end{align}
        The Laplace transform of the solution to \eqref{intr1} reads
        \begin{equation}
            \int_0^{+\infty}e^{-st}p_k^{\nu_k}(t) \mathrm dt
            =\left(\prod_{j=1}^{k-1}\lambda_j\right) \frac{s^{\nu_1-1}}{\prod_{j=1}^k
            (s^{\nu_j}+\lambda_j)}.
        \end{equation}
        A similar and more general state-dependent fractional
        birth-death process was recently tackled by \citet{Fedotov},
        where possible applications to chemotaxis are sketched.\\
        The case where $\nu_k = \nu$, for all $k$ in \eqref{intr1}, has been dealt with in \citet{fede}.
        An attempt to apply this fractional birth process was discussed in \citet{phys}
        in relation to tumoral growth models and ETAS (\textit{Epidemic Type Aftershock Sequences})
        model in statistical seismology.\\
        The dependence of the state probabilities of the point
        processes considered here from the structure of $\nu_k$, requires a further investigation
        which certainly implies a numerical approach.

        \subsection{Notation}

            For the sake of clarity we briefly summarize the notation used for the different
            point processes analyzed in the following sections.

            First we indicate with $N(t)$, $t \ge 0$, the counting process associated with the
            variable-order difference-differential equations \eqref{aao}. In particular,
            the state probabilities $p_k^{\nu_k}(t) = \Pr \{ N(t) = k \}$, $k \ge 0$, represent the
            probability of being in state $k$ at a fixed time $t \ge 0$.
            The point process constructed and studied in Section \ref{terzo}
            by means of independent but non i.d.\ inter-arrival times
            is instead indicated by a calligraphic $\mathcal{N}(t)$.
            Both processes, when $\nu_k=\nu$ for all $k \ge 0$, reduce to the
            time-fractional Poisson process $N_\nu(t)$ treated for example in
            \citet{Beghin}. In the same article it is also considered the alternative
            definition for a fractional Poisson process characterized by the distribution
            \eqref{palomar} and denoted here by $\hat{N}_\nu(t)$, $t \ge 0$.
            We refer to its direct generalization in a state-dependent sense
            as $\hat{N}(t)$ for which the distribution becomes that in \eqref{palombaro}.
            Lastly, the linear fractional pure birth process with state-dependent
            order of fractionality presented in the last section
            is simply indicated as $N_{lin}(t)$, $t \ge 0$.

    \section{The state-dependent fractional Poisson process}

        \label{due}
        We first consider a state-dependent time-fractional Poisson process
        $N(t)$, $t\geq 0$, whose state probabilities $p_k^{\nu_k}(t)=\Pr\{N(t)=k\}$
        are governed by  equations \eqref{aao}.
        We have the following result

        \begin{te}
            The Laplace transform of the solution to the state-dependent
            time-fractional equations
            \begin{align}
                \label{aa}
                \begin{cases}
                    \frac{\mathrm{d}^{\nu_k}}{\mathrm{d}t^{\nu_k}} p_k^{\nu_k}(t)
                    = -\lambda p_k^{\nu_k}(t) +\lambda p_{k-1}^{\nu_{k-1}}(t), & k \ge 0, \: t > 0, \: \nu_k \in (0,1], \\
                    p_k^{\nu_k}(0) =
                    \begin{cases}
                        1, \quad k=0, \\
                        0, \quad k \geq 1,
                    \end{cases}
                \end{cases}
            \end{align}
            reads
            \begin{align}
                \tilde{p}_k^{\nu_k} (s) = \int_0^{+\infty}e^{-st}p_k^{\nu_k}(t) \mathrm dt
                = \frac{\lambda^k s^{\nu_0-1}}{\prod_{j=0}^k
                (s^{\nu_j}+\lambda)},
            \end{align}
            where the fractional derivative appearing in \eqref{aa} is
            in the sense of Dzhrbashyan--Caputo.

            \begin{proof}
                We can solve equation \eqref{aa} by means of an iterative
                procedure, as follows.
                The equation related to $k=0$
                \begin{align}
                    \label{aa0}
                    \begin{cases}
                        \frac{\mathrm{d}^{\nu_0}}{\mathrm{d}t^{\nu_0}} p_0^{\nu_0}(t)
                        = -\lambda p_0^{\nu_0}(t), & \: t > 0, \: \nu_0 \in (0,1], \\
                        p_0^{\nu_0}(0) = 1,
                    \end{cases}
                \end{align}
                has solution
                $p_0^{\nu_0}(t)= E_{\nu_0,1}(-\lambda t^{\nu_0}),$
                with Laplace transform
                \begin{equation}
                    \tilde{p}_0^{\nu_0} (s)=\int_0^{+\infty}e^{-st}p_0^{\nu_0}(t) \mathrm dt =
                    \frac{s^{\nu_0-1}}{\lambda+s^{\nu_0}},
                \end{equation}
                where
                \begin{align}
                    E_{\nu_0,1}(-\lambda t^{\nu_0})= \sum_{k=0}^{\infty}
                    \frac{(-\lambda t^{\nu_0})^k}{\Gamma(\nu_0 k+1)},
                \end{align}
                is the Mittag--Leffler function.

                For $k=1$, the equation
                \begin{align}
                    \label{aa1}
                    \begin{cases}
                        \frac{\mathrm{d}^{\nu_1}}{\mathrm{d}t^{\nu_1}} p_1^{\nu_1}(t)
                        = -\lambda p_1^{\nu_1}(t) + \lambda p_0^{\nu_0}(t),
                        & \: t > 0, \: \nu_1 \in (0,1], \\
                        p_1^{\nu_1}(0) = 0,
                    \end{cases}
                \end{align}
                has solution with Laplace transform
                \begin{equation}\label{ogg}
                    \tilde{p}_1^{\nu_1} (s)=\int_0^{+\infty}e^{-st}p_1^{\nu_1}(t) \mathrm dt =
                    \frac{\lambda s^{\nu_0-1}}{\lambda+s^{\nu_0}}\frac{1}{\lambda+s^{\nu_1}}.
                \end{equation}
                By iterating this procedure, we arrive at
                \begin{align}
                    \label{casa}
                    \tilde{p}_k^{\nu_k} (s) = \int_0^{+\infty}e^{-st}p_k^{\nu_k}(t) \mathrm dt
                    = \frac{\lambda^k s^{\nu_0-1}}{\prod_{j=0}^k (s^{\nu_j}+\lambda)}.
                \end{align}
            \end{proof}
        \end{te}

        \begin{os}
            A direct approach based on the inversion of the Laplace
            transform of \eqref{casa} is clumsy and cumbersome. We give the
            explicit evaluation of $p_1^{\nu_1} (t)$.
            In this case, from \eqref{ogg}, we have that
            \begin{align}
                \label{conto}
                \tilde{p}_1^{\nu_1} (s) & =\int_0^{+\infty}e^{-st}p_1^{\nu_1}(t) \mathrm dt \\
                & = \frac{\lambda s^{\nu_0-1}}{\lambda^2+\lambda
                (s^{\nu_0}+s^{\nu_1})+s^{\nu_0+\nu_1}} \notag \\
                \nonumber &=\frac{\lambda s^{\nu_0-1}}{\lambda^2+s^{\nu_0+\nu_1}}\frac{1}{1+\frac{\lambda
                (s^{\nu_0}+s^{\nu_1})}{\lambda^2+s^{\nu_0+\nu_1}}}\\
                \nonumber &=\lambda s^{\nu_0-1}\sum_{m=0}^{\infty}\frac{(-\lambda
                (s^{\nu_0}+s^{\nu_1}))^m}{(\lambda^2+s^{\nu_0+\nu_1})^{m+1}}\\
                \nonumber &=\lambda
                s^{\nu_0-1}\sum_{m=0}^{\infty}
                \frac{(-\lambda)^m}{(\lambda^2+s^{\nu_0+\nu_1})^{m+1}}\sum_{r=0}^{m}\binom{m}{r}
                s^{\nu_0 r+\nu_1(m-r)}.
             \end{align}

            The inversion of \eqref{conto} involves the generalized
            Mittag--Leffler function, defined as (see, for example, \citet{Saxena})
            \begin{equation}
                E_{\nu,\beta}^m(-\lambda t^{\nu})= \sum_{k=0}^{\infty}\frac{(-\lambda t^{\nu})^k\Gamma(m+k)}{k!\Gamma(\nu
                k+\beta)\Gamma(m)},
            \end{equation}
            where $\nu, \beta, m\in\mathbb{R}^+$.

            Indeed, we recall the following relation.
            \begin{equation}
                \label{rico}
                \int_0^{+\infty}e^{-st}t^{\beta-1}E_{\nu,\beta}^m(-\lambda
                t^{\nu}) \mathrm dt = \frac{s^{\nu m-\beta}}{(\lambda+s^{\nu})^m}.
            \end{equation}
            In view of \eqref{conto} and \eqref{rico}, we arrive at
            \begin{equation}\label{sax}
                p_1^{\nu_1} (t)=\sum_{m=0}^{\infty}(-1)^m \lambda^{m+1}\sum_{r=0}^{m}\binom{m}{r}
                t^{\nu_0(m-r)+\nu_1 r+\nu_1}E_{\nu_0+\nu_1,\nu_0(m-r)+\nu_1
                r+\nu_1+1}^{m+1}(-\lambda^2
                t^{\nu_0+\nu_1}).
            \end{equation}
            For the case $\nu_1=\nu_0=\nu$, formula \eqref{sax} becomes
            \begin{align}\label{sax1}
                p_1^{\nu} (t)&=\sum_{m=0}^{\infty}(-1)^m \lambda^{m+1}\sum_{r=0}^{m}\binom{m}{r}
                t^{\nu (m +1)}E_{2\nu,\nu(m+1)+1}^{m+1}(-\lambda^2
                t^{2\nu})\\
                \nonumber &= \sum_{m=0}^{\infty}(-1)^m (\lambda t^{\nu})^{m+1}2^m
                E_{2\nu,\nu(m+1)+1}^{m+1}(-\lambda^2 t^{2\nu})\\
                \nonumber &= \sum_{m=0}^{\infty}(-1)^m (\lambda t^{\nu})^{m+1}2^m
                \sum_{r=0}^{\infty}\binom{m+r}{r}\frac{(-1)^r(\lambda^2 t^{2\nu})^r}{\Gamma(2\nu r+\nu(m+1)+1)}\\
                \nonumber &= \sum_{m=0}^{\infty}(-1)^m (\lambda t^{\nu})^{m+1}2^m
                \sum_{r=0}^{\infty}\binom{-(m+1)}{r}(\lambda^2 t^{2\nu})^r\frac{1}{2\pi i}
                \int_{Ha}e^w w^{-2\nu r-\nu(m+1)-1}\mathrm dw\\
                \nonumber &= \sum_{m=0}^{\infty}(-1)^m (\lambda t^{\nu})^{m+1}2^m \frac{1}{2\pi i}
                \int_{Ha}e^w w^{-\nu(m+1)-1}
                \left[\sum_{r=0}^{\infty}\binom{-(m+1)}{r}(\lambda^2 t^{2\nu}w^{-2\nu})^r\right] \mathrm dw\\
                \nonumber &= \sum_{m=0}^{\infty}(-1)^m \lambda t^{\nu})^{m+1}2^m
                \frac{1}{2\pi i}\int_{Ha}e^w
                \frac{w^{-\nu(m+1)-1}}{(\lambda^2 t^{2\nu}w^{-2\nu}+1)^{m+1}}\mathrm dw  \\
                \nonumber &= \frac{1}{2\pi i}\int_{Ha} e^w \frac{\lambda t^{\nu}w^{-\nu-1}}{\lambda^2 t^{2\nu}w^{-2\nu}+1}
                \left[\sum_{m=0}^{\infty}(-1)^m \left(\frac{2 w^{-\nu}\lambda
                t^{\nu}}{\lambda^2 t^{2\nu}w^{-2\nu}+1}\right)^{m}\right]\mathrm dw\\
                \nonumber &= \frac{1}{2\pi i}\int_{Ha}\frac{\lambda t^{\nu}w^{\nu-1}e^w}{(w^{\nu}+\lambda t^{\nu})^2}
                \mathrm dw\\
                \nonumber &= \frac{\lambda t^{\nu}}{\nu}E_{\nu,\nu}(-\lambda t^{\nu}),
            \end{align}
            where we have used in the last equality the fact that
            \begin{align}
                E_{\nu, \nu}(x)=\nu\frac{\mathrm d}{\mathrm dx}E_{\nu, 1}(x)&= \frac{\nu}{2\pi i}
                \frac{\mathrm d}{\mathrm dx}\int_{Ha}
                \frac{e^w w^{\nu-1}}{w^{\nu}-x} \mathrm dw
                = \frac{\nu}{2\pi i}\int_{Ha}\frac{e^w w^{\nu-1}}{(w^{\nu}-x)^2}\mathrm dw,
            \end{align}
            and we have applied the contour-integral representation of the reciprocal of the Gamma function
            \begin{equation}
                \frac{1}{\Gamma(x)}=\frac{1}{2\pi i}\int_{Ha}e^u u^{-x}\mathrm du,
            \end{equation}
            where $Ha$ stands for the Hankel contour (see formula 5.9.2, pg. 139 in \citet{olv}).

            We notice that equation \eqref{sax1} gives the result obtained for the time-fractional
            Poisson process in \citet{Beghin} as expected.
            Moreover by considering that
            \begin{equation}
                \int_0^{+\infty}e^{-st}\frac{\lambda
                t^{\nu}}{\nu}E_{\nu,\nu}(-\lambda
                t^{\nu}) \mathrm dt=\frac{\lambda s^{\nu-1}}{(\lambda+s^{\nu})^2},
            \end{equation}
            we retrieve, for the case $\nu = \nu_0 = \nu_1$ that
            \begin{align}
                p_1^{\nu} (t)=\frac{\lambda t^{\nu}}{\nu}E_{\nu,\nu}(-\lambda
                t^{\nu}),
            \end{align}
            that is the result obtained for the time-fractional Poisson process
            (see formula (2.11) of \citet{Beghin}).

            By applying formula (34) of \citet{Saxena} it is possible to give an explicit expression for
            $p_k^{\nu_k} (t)$, for any $k\geq 2$, in terms of
            cumbersome sums of generalized Mittag--Leffler functions.
        \end{os}

        \begin{os}
            A different way to give a representation of the state
            probability in the state-dependent Poisson process is given by
            the following integral approach; starting from \eqref{casa}, we
            have
            \begin{align}
                \label{casai3}
                \tilde{p}_k^{\nu_k} (s) & = \int_0^{+\infty}e^{-st}p_k^{\nu_k}(t) \mathrm dt
                = \lambda^k \frac{s^{\nu_0-1}}{\prod_{j=0}^k
                (s^{\nu_j}+\lambda)} \\
                & = \left( \int_0^\infty e^{-\lambda w_0} s^{\nu_0-1} e^{-w_0s^{\nu_0}} \mathrm dw_0 \right)
                \left( \prod_{j=1}^k \int_0^\infty e^{-\lambda w_j} \lambda e^{-w_j s^{\nu_j}} \mathrm dw_j
                \right). \notag
            \end{align}
            For the following developments, it is useful to recall that the inverse
            process of a $\nu$-stable subordinator $\mathcal{H}^{\nu}(t)$,
            $t \ge 0$, namely $\mathcal{L}^{\nu}(t)$, $t \ge 0$, is such that
            \begin{equation}
                \Pr\{\mathcal{L}^{\nu}(t)<x\} = \Pr\{\mathcal{H}^{\nu}(x)>t\},\qquad
                x, \: t \ge 0.
            \end{equation}
            Hence the relation between the law $l_{\nu}(x,t)$ of the process
            $\mathcal{L}^{\nu}(t)$ and the law $h_{\nu}(x,t)$ of the process
            $\mathcal{H}^{\nu}(t)$ is given by (see for example
            \citet{Mirko})
            \begin{equation}
                l_{\nu}(x,t)=\frac{\Pr\{\mathcal{L}^{\nu}(t)\in \mathrm dx\}}{\mathrm dx}=\frac{\partial}{\partial
                x}\Pr\{\mathcal{H}^{\nu}(x)>t\}= \frac{\partial}{\partial
                x}\int_t^{\infty} h_{\nu}(s,x) \mathrm ds,
            \end{equation}
            or otherwise
            \begin{equation}
                \int_t^{\infty} \Pr\{\mathcal{H}^{\nu}(x)\in \mathrm dw\}=\int_0^{x} \Pr\{\mathcal{L}^{\nu}(t)\in
                \mathrm dz\}.
            \end{equation}
            Hence the density of the inverse process $\mathcal{L}^{\nu}(t)$
            reads
            \begin{equation}
                \Pr\{\mathcal{L}^{\nu}(t)\in
                \mathrm dx\}=\frac{\partial}{\partial x}\int_t^{\infty}\Pr\{\mathcal{H}^{\nu}(x)\in
                \mathrm dw\}.
            \end{equation}
            Therefore the Laplace transform of $l_{\nu}(x,t)$ is given by
            \begin{align}\label{inver}
                \tilde {l}_{\nu}(x,s)= \int_0^\infty e^{-st}l_{\nu}(x,t) \mathrm dt
                & =\int_0^\infty e^{-st}
                \frac{\mathrm d}{\mathrm dx}\left[\int_t^{+\infty}\Pr\{\mathcal{H}^{\nu}(x)\in \mathrm dw\}\right]
                \mathrm dt\\
                \nonumber &=\frac{\mathrm d}{\mathrm dx}\int_0^\infty \Pr\{\mathcal{H}^{\nu}(x)\in
                \mathrm dw\}\int_0^w e^{-st} \mathrm dt\\
                \nonumber &=\frac{1}{s}\frac{\mathrm d}{\mathrm dx}\left[\int_0^\infty (1-e^{-sw})
                \Pr\{\mathcal{H}^{\nu}(x)\in
                \mathrm dw\} \right]= s^{\nu-1}e^{-xs^{\nu}},
            \end{align}
            where we used the fact that
            \begin{align}
                \tilde{h}_{\nu}(x,s)=\int_0^{+\infty}e^{-st}h_{\nu}(x,t) \mathrm dt=e^{-xs^{\nu}}.
            \end{align}
            We also notice that the explicit form of the law of the inverse
            of the stable subordinator is known in terms of Wright
            functions \citep{Mirko}.
            Going back to equation \eqref{casai3} and in view of \eqref{inver}, we can write
            \begin{align}
                \label{casai2}
                \tilde{p}_k^{\nu_k} (s) & = \left( \int_0^\infty e^{-\lambda w_0} \mathrm dw_0
                \int_0^\infty e^{-st} l_{\nu_0}(w_0,t) \mathrm dt \right) \left( \prod_{j=1}^k
                \lambda \int_0^\infty e^{-\lambda w_j} \mathrm dw_j \int_0^\infty
                e^{-sx} h_{\nu_j} (x,w_j) \mathrm dx \right)\\
                & = \int_0^\infty \mathrm dw_0 e^{-\lambda w_0}\cdots\int_0^\infty \mathrm dw_k e^{-\lambda w_k} \left[
                \int_0^\infty e^{-st} l_{\nu_0}(w_0,t) \mathrm dt
                \prod_{j=1}^k\lambda \int_0^\infty e^{-sx} h_{\nu_j} (x,w_j) \mathrm dx\right]. \notag
            \end{align}
            Hence by inverting the Laplace transform we obtain
            \begin{align}
                \label{casai1}
                p_k^{\nu_k} (t) = \lambda^k \int_0^{\infty} \mathrm dw_0 e^{-\lambda w_0}\int_0^\infty
                \mathrm dw_1 e^{-\lambda w_1}\cdots\int_0^\infty \mathrm
                dw_k e^{-\lambda w_k}\left[l_{\nu_0}(w,t)\ast h_{\nu_1,\cdots, \nu_k}(w_1,\cdots, w_k,t)\right],
            \end{align}
            where the symbol $\ast$ stands for the convolution of the law of the inverse stable subordinator
            $l_{\nu_0}$ and the distribution of the sum of $k$ independent stable subordinators
            $h_{\nu_1,\cdots, \nu_k}(w_1,\cdots, w_k,t)$. In other words
            $l_{\nu_0}(w,t)\ast h_{\nu_1,\cdots, \nu_k}(w_1,\cdots, w_k,t)$ is the distribution of the r.v.
            \begin{align}
                \mathcal{L}^{\nu_0}(t)+\sum_{j=1}^{k}\mathcal{H}^{\nu_j}(t).
            \end{align}
        \end{os}

        \begin{os}
            Another interesting characterization of the state-probabilities of
            the above process is given by the following observation.
            First of all, since for $m=1$, $E^1_{\nu, \beta}(\cdot)=E_{\nu, \beta}(\cdot)$, from \eqref{rico} we have that
            \begin{align}
                \label{ricos}
                &\int_0^{+\infty}e^{-st}t^{\nu-1}E_{\nu,\nu}(-\lambda
                t^{\nu})\mathrm dt=\frac{1}{\lambda+s^{\nu}},\\
                &\int_0^{+\infty}e^{-st}E_{\nu,1}(-\lambda
                t^{\nu})\mathrm dt=\frac{s^{\nu-1}}{\lambda+s^{\nu}}.
            \end{align}
            Hence, from \eqref{casa}, we find that
            \begin{align}
                \label{ML}
                \tilde{p}_k^{\nu_k} (s)&=\frac{\lambda^k s^{\nu_0-1}}{\prod_{j=0}^k (s^{\nu_j}+\lambda)}\\
                \nonumber &= \left[\int_0^{+\infty}e^{-st}E_{\nu_0,1}(-\lambda
                t^{\nu})\mathrm dt\right] \prod_{j=1}^k\left[\int_0^{+\infty}e^{-st}\lambda t^{\nu_j-1}
                E_{\nu_j,\nu_j}(-\lambda
                t^{\nu_j})\mathrm dt\right].
            \end{align}
            On the other hand, from \eqref{casai2}, we have
            \begin{align}
                \label{cara}
                \tilde{p}_k^{\nu_k} (s) & = \left( \int_0^\infty e^{-\lambda w_0} \mathrm dw_0
                \int_0^\infty e^{-st} l_{\nu_0}(w_0,t) \mathrm dt \right) \left( \prod_{j=1}^k
                \lambda \int_0^\infty e^{-\lambda w_j} \mathrm dw_j \int_0^\infty e^{-sx} h_{\nu_j} (x,w_j)
                \mathrm dx \right) \\
                &=\left( \int_0^\infty e^{-st} \tilde{l}_{\nu_0} (\lambda,t) \mathrm dt \right)
                \left( \prod_{j=1}^k \lambda \int_0^\infty e^{-sx} \tilde{h}_{\nu_j}
                (x,\lambda) \mathrm dx \right),     \notag
            \end{align}
            which clearly coincides with \eqref{ML}.

            By inverting the Laplace transform, we obtain the following
            result
            \begin{align}
                \label{praba}
                p_k^{\nu_k} (t)&=  E_{\nu_0,1}(-\lambda t^{\nu_0})
                \bigast_{j=1}^k \lambda t^{\nu_j-1}E_{\nu_j,\nu_j}(-\lambda
                t^{\nu_j})\\
        \nonumber &= \int_0^{\infty}E_{\nu_0,1}\left(-\lambda(t-s)^{\nu_0}\right)g(s)ds,
            \end{align}
        where $g(s)$ is the $k-$th time iterated convolution of the functions
        \begin{equation}\nonumber
        h_j(t)= \lambda t^{\nu_j-1}E_{\nu_j, \nu_j}\left(-\lambda t^{\nu_j}\right)
        \end{equation}
            We notice that, the last equation can be written in terms of the Prabhakar operator,
            that is an integral operator involving a Mittag--Leffler function
            as kernel \citep{Prab}. From equation \eqref{praba} we have an integral representation,
            in explicit form given by
            \begin{align}
                &p_1^{\nu_1}(t)=\int_0^t E_{\nu_0,1}(-\lambda(t-s)^{\nu_0})E_{\nu_1,\nu_1}
                (-\lambda s^{\nu_1})s^{\nu_1-1}ds\\
                \nonumber &p_2^{\nu_2}(t)= \int_0^t ds_1 E_{\nu_0,1}(-\lambda(t-s_1)^{\nu_0})
                \int_0^{s_1}ds_2 s_2^{\nu_1-1}E_{\nu_1,\nu_1}(-\lambda s_2^{\nu_1})(s_1-s_2)^{\nu_2-1}
                E_{\nu_2,\nu_2}(-\lambda (s_1-s_2)^{\nu_2})\\
                \nonumber &\vdots\\
                \nonumber &p_k^{\nu_k}(t)= \int_0^t ds_1 E_{\nu_0,1}(-\lambda(t-s_1)^{\nu_0})
                \int_0^{s_1}ds_2\dots \times \\
                & \qquad \qquad \times \int_0^{s_k}ds_k s_k^{\nu_{k-1}-1}E_{\nu_{k-1},\nu_{k-1}}
                (-\lambda s_k^{\nu_{k-1}})(s_{k-1}-s_k)^{\nu_k-1}
                E_{\nu_k,\nu_k}(-\lambda (s_{k-1}-s_k)^{\nu_k}). \notag
            \end{align}
        \end{os}

        In order to find the mean value of the distribution $p_k^{\nu_k} (t)$, we multiply all the terms of
        \eqref{aa} for $k$ and sum over all the states
        so that
        \begin{align}
            \label{sum}
            \sum_{k=0}^{\infty}k\frac{\mathrm{d}^{\nu_k}}{\mathrm{d}t^{\nu_k}} p_k^{\nu_k}(t)
            &= -\lambda \sum_{k=0}^{\infty}k p_k^{\nu_k}(t) +
            \lambda \sum_{k=0}^{\infty}k p_{k-1}^{\nu_{k-1}}(t)\\
            \nonumber &= -\lambda \sum_{k=0}^{\infty}k p_k^{\nu_k}(t) +
            \lambda \sum_{k=0}^{\infty}(k+1)p_{k}^{\nu_{k}}(t)= \lambda.
        \end{align}
        In the case $\nu_k = \nu$, for all $k$, we have
        \begin{equation}
            \frac{\mathrm{d}^{\nu}}{\mathrm{d}t^{\nu}}\sum_{k=0}^{\infty}k p_k^{\nu}(t)
            = \frac{\mathrm{d}^{\nu}}{\mathrm{d}t^{\nu}}\mathbb{E}(N_{\nu}(t))=\lambda,
        \end{equation}
        whose solution is given by
        $\mathbb{E}(N_{\nu}(t))=\frac{\lambda t^{\nu}}{\Gamma(\nu+1)}$ (see formula (2.7) in \citet{Beghin}).
        We notice that it is possible to find an interesting summation formula by using the Laplace transform
        in equation \eqref{sum}.
        Indeed we have
        \begin{align}
            \sum_{k=0}^{\infty}k s^{\nu_k}\tilde{p}_k^{\nu_k}(s)&= \lambda s^{-1},
        \end{align}
        and recalling that
        \begin{align}\label{dome}
            \tilde{p}_k^{\nu_k} (s) = \frac{\lambda^{k} s^{\nu_0-1}}{\prod_{j=0}^k (s^{\nu_j}+\lambda)},
        \end{align}
        we find that
        \begin{equation}
            \sum_{k=0}^{\infty} \frac{k \lambda^{k} s^{\nu_0+\nu_k}}{\prod_{j=0}^k (s^{\nu_j}+\lambda)} = \lambda.
        \end{equation}
        This summation formula is not trivial and we can check that it works for example in the special case
        $\nu = \nu_k$ for all $k$.
        \begin{align}
            \sum_{k=1}^{\infty}\frac{k \lambda^k s^{2\nu}}{(s^{\nu}+\lambda)^{k+1}}
            &= \frac{s^{2\nu}}{s^{\nu}+\lambda}\sum_{k=1}^{\infty}\frac{k \lambda^k}
            {(s^{\nu}+\lambda)^{k}}   \\
            \nonumber &= \frac{\lambda s^{2\nu}}{s^{\nu}+\lambda}\left[\frac{d}{dw}\sum_{k=1}^{\infty}
            \frac{w^k}{(s^{\nu}+\lambda)^k}\right]_{w=\lambda}\\
            \nonumber &= \frac{\lambda s^{2\nu}}{s^{\nu}+\lambda}\left[\frac{d}{dw}
            \frac{w}{s^{\nu}+\lambda-w}\right]_{w=\lambda}\\
            \nonumber &= \frac{\lambda s^{2\nu}}{s^{\nu}+\lambda}\left[\frac{s^{\nu}
            +\lambda}{(s^{\nu}+\lambda-w)^2}\right]_{w=\lambda}
            \nonumber = \lambda
        \end{align}

        \begin{os}
            We notice that for the probability generating function $G(u,t)$
            of the process $N(t)$, $t\geq 0$, the following
            representation holds for $u\in [0,1]$
            \begin{align}\label{gen}
                \int_0^{\infty}e^{-st}G(u,t)\mathrm dt&=\sum_{k=0}^{\infty} u^k
                \int_0^{\infty} e^{-st}\Pr\{N(t)=k\} \mathrm dt\\
                \nonumber &=\int_0^{\infty}e^{-st}\Pr\{\min_{0\leq k \leq N(t)}X_k>
                1-u\} \mathrm dt\\
                \nonumber &= \sum_{k=0}^{\infty}\frac{\lambda^k u^k s^{\nu_0-1}}{\prod_{j=0}^{k}(\lambda+
                s^{\nu_j})},
            \end{align}
            where $X_k$, $k\geq 1$, are i.i.d.\ random variables uniform
            in $[0,1]$.

            The representation of the probability generating function as
            \begin{equation}
                G(u,t)= \Pr\{\min_{0\leq k \leq N(t)}X_k>
                1-u\},
            \end{equation}
            follows the same lines of the time and space fractional
            Poisson processes described in \citet{fed}. In \eqref{gen},
            the driving process is the state-dependent Poisson process.
        \end{os}

    \section{Alternative forms of the state-dependent Poisson process}

        \label{terzo}
        We construct now a point process with independent but not i.d.\ inter-arrival times. In
        particular,
        the waiting time $U_k$ between the $k$th and $(k+1)$th arrival is distributed with p.d.f.\
        \begin{align}
            f_{U_k}(t) = \lambda t^{\nu_k-1} E_{\nu_k,\nu_k} (-\lambda t^{\nu_k}), \; t>0
        \end{align}
        Let us now call $\mathcal{N}(t)$, $t \ge 0$, such process and we have the following theorem.
        \begin{te}
            The state probabilities $p_k(t)$ of the process $\mathcal{N}(t)$, $t \ge
            0$, are governed by the integral equation
            \begin{equation}
                p_k(t)-p_k(0)= -\lambda I^{\nu_k} p_k(t)+ \lambda I^{\nu_{k-1}}
                p_{k-1}(t),\qquad t \ge 0, \: \nu_k \in (0,1],
            \end{equation}
            where $I^{\nu}$ is the fractional integral in the sense of
            Riemann--Liouville (see \eqref{riml}).
            Moreover, their Laplace transforms are given by
            \begin{equation}
                \int_0^\infty e^{-s t} \Pr \{ \mathcal{N}(t) = k \}
                \mathrm{d}t= \lambda^k \frac{s^{\nu_k-1}}{\prod_{j=0}^k (s^{\nu_j}+\lambda)}.
            \end{equation}

            \begin{proof}
                First, we observe that the Laplace transform of the state
                probabilities, can be directly calculated by using the definition of the process
                $\mathcal{N}(t)$
                \begin{align}\label{ren}
                    & \int_0^\infty e^{-s t} \Pr \{ \mathcal{N}(t) = k \} \mathrm{d}t \\
                    & = \int_0^\infty e^{-st} \mathrm{d}t \left[ \int_0^t \Pr (U_0+\dots +U_{k-1} \in \mathrm{d}y)
                    - \int_0^t \Pr (U_0+\dots +U_k \in \mathrm{d}y) \right] \notag \\
                    & = \frac{1}{s} \int_0^\infty e^{-sy} \left[ \Pr (U_0+\dots +U_{k-1} \in \mathrm{d}y)
                    - \Pr (U_0+\dots +U_k \in \mathrm{d}y) \right] \notag \\
                    & = \frac{1}{s} \left[ \frac{\lambda^k}{\prod_{j=0}^{k-1} (\lambda+s^{\nu_j})}
                    - \frac{\lambda^{k+1}}{\prod_{j=0}^k (\lambda + s^{\nu_j})} \right] \notag \\
                    & = \frac{1}{s} \frac{\lambda^k (\lambda+s^{\nu_j})-\lambda^{k+1}}{\prod_{j=0}^k
                    (\lambda+s^{\nu_k})} \notag \\
                    & = \lambda^k \frac{s^{\nu_k-1}}{\prod_{j=0}^k (s^{\nu_j}+\lambda)}. \notag
                \end{align}
                We notice that, unfortunately, it does not coincide with \eqref{casa}.

                Hence we have two distinct processes that can be matched only by
                assuming that $\nu_k=\nu$ for each $k=0,1\dots$
                (in other words in the time-fractional Poisson case).

                We can also find in explicit way the integral equation governing
                the probabilities $p_k(t)=\Pr \{ \mathcal{N}(t) = k \}$.
                We start from the ordinary difference-differential equation,
                governing the Poisson process
                \begin{equation}
                    \frac{\mathrm d p_k}{\mathrm dt}(t)= -\lambda p_k(t)+\lambda p_{k-1}(t),
                \end{equation}
                with initial conditions
                \begin{align}
                    p_k(0) =
                    \begin{cases}
                        1 \quad k=0, \\
                        0 \quad k \geq 1.
                    \end{cases}
                \end{align}
                By integration with respect to $t$, we have the equivalent integral
                equation
                \begin{equation}
                    \label{rl}
                    p_k(t)-p_k(0)= -\lambda \int_0^t p_k(s) \mathrm ds+\lambda\int_0^t
                    p_{k-1}(s) \mathrm ds,
                \end{equation}

                In order to obtain a fractional generalization of the last equation, we
                replace the first-order integral in the right hand side of \eqref{rl}, with state-dependent
                fractional integrals, i.e.
                \begin{equation}
                    \label{rl1}
                    p_k(t)-p_k(0)= -\lambda I^{\nu_k} p_k(t)+ \lambda I^{\nu_{k-1}}
                    p_{k-1}(t),\quad t \ge 0, \: \nu_k \in (0,1], \: k \ge 0,
                \end{equation}
                where $I^{\nu_k}$ is the fractional integral in the sense of
                Riemann--Liouville.
                For $k=0$, we have
                \begin{equation}
                    p_0(t)-1= -\lambda I^{\nu_0} p_0(t),
                \end{equation}
                whose solution is simply given by
                $p_0(t)= E_{\nu_0,1}(-\lambda t^{\nu_0})$.
                With $k=1$, we obtain
                \begin{equation}
                    p_1(t)= -\lambda I^{\nu_1} p_1(t)+ \lambda I^{\nu_{0}}
                    p_{0}(t),
                \end{equation}
                whose Laplace transform, after some simple calculation, is given by
                \begin{equation}
                    \tilde{p}_1(t)= \frac{\lambda s^{\nu_1-1}}{(s^{\nu_0}+\lambda)(s^{\nu_1}+\lambda)},
                \end{equation}
                and coincides with \eqref{ren} in the case $k=1$. Then, it is
                immediate to prove that, for any order $k\geq 1$, the Laplace
                transform of $p_k(t)$, is given by \eqref{ren}. This proves that
                \eqref{rl1} is the governing equation for $\mathcal{N}(t)$, $t \ge 0$,
                as claimed.
            \end{proof}
        \end{te}

        In order to highlight the relation between the two processes
        $N(t)$, $t \ge 0$, and $\mathcal{N}(t)$, $t \ge 0$, by rearranging \eqref{dome}, we can write the following
        \begin{align}
            \tilde{p}_k^{\nu_k}(s) = s^{\nu_0-\nu_k} \frac{\lambda^k s^{\nu_k-1}}{\prod_{j=0}^k (s^{\nu_j}+\lambda)}.
        \end{align}
        Therefore if $(\nu_k-\nu_0) > 0$ for a fixed $k$ we have that
        \begin{align}
            p_k^{\nu_k} (t) & = \frac{1}{\Gamma(\nu_k-\nu_0)} \int_0^t (t-y)^{(\nu_k-\nu_0)-1} \Pr \{ \mathcal{N}(y)=k \}
            \, \mathrm dy \\
            & = I^{\nu_k-\nu_0} \Pr \{ \mathcal{N}(t)=k\}, \qquad t \ge 0, \notag
        \end{align}
        where $I^{\nu_k-\nu_0}$ is the Riemann--Liouville fractional integral. Note that since the Riemann--Liouville
        fractional derivative (that we indicate here with $D^\alpha$)
        is the left-inverse operator to the Riemann--Liouville fractional integral
        we also obtain the related relation
        \begin{align}
            D^{\nu_k-\nu_0} p_k^{\nu_k} (t) = \Pr \{ \mathcal{N}(t)=k\}, \qquad t \ge 0, \: (\nu_k-\nu_0)> 0.
        \end{align}

        Conversely, in view of \eqref{ren}, we can write
        \begin{align}
            \int_0^\infty e^{-s t} \Pr \{ \mathcal{N}(t) = k \} \mathrm{d}t
            = s^{\nu_k-\nu_0} \frac{\lambda^k s^{\nu_0-1}}{\prod_{j=0}^k (s^{\nu_j}+\lambda)},
        \end{align}
        and thus if $(\nu_0 -\nu_k)> 0$, for a fixed $k$, we obtain that
        \begin{align}
            \Pr \{ \mathcal{N}(t) = k \}
            & = \frac{1}{\Gamma(\nu_0-\nu_k)} \int_0^t (t-y)^{(\nu_0-\nu_k)-1} p_k^{\nu_k}(y)
            \, \mathrm dy \\
            & = I^{\nu_0-\nu_k} p_k^{\nu_k}(t), \qquad t \ge 0, \notag
        \end{align}
        and that
        \begin{align}
            D^{\nu_0-\nu_k} \Pr \{ \mathcal{N}(t) = k \} = p_k^{\nu_k}(t),
            \qquad t \ge 0, \: (\nu_0-\nu_k)> 0.
        \end{align}
        Finally, we have the following relation between the state probabilities of the two processes
        \begin{equation}
            \label{rel}
            p_k^{\nu_k}(t)=
            \begin{cases}
                I^{\nu_k-\nu_0} \Pr \{ \mathcal{N}(t)=k\},\quad &\nu_k > \nu_0,\\
                D^{\nu_0-\nu_k} \Pr \{ \mathcal{N}(t)=k\}, \quad &\nu_k < \nu_0.
            \end{cases}
        \end{equation}
        In order to deepen the meaning of this relation, we consider as an example the relation
        between $p_1^{\nu_1}(t)$ and $\Pr \{ \mathcal{N}(t)= 1\}$.

        By inverting the Laplace transform \eqref{ren}, we obtain
        that
        \begin{equation}
            \Pr \{ \mathcal{N}(t)=1\}=\sum_{m=0}^{\infty}(-1)^m \lambda^{m+1}\sum_{r=0}^{m}\binom{m}{r}
            t^{\nu_0(m-r)+\nu_1 r+\nu_0}E_{\nu_0+\nu_1,\nu_0(m-r)+\nu_1
            r+\nu_0+1}^{m+1}(-\lambda^2
            t^{\nu_0+\nu_1}),
        \end{equation}
        by calculation similar to those given above for $p_1^{\nu_1}(t)$.
        Recalling that (\citet{Sax}, page 123)
        \begin{equation}
            I^{\alpha}[t^{\gamma-1}E^{m}_{\beta, \gamma}(at^{\beta})]=t^{\alpha+\gamma-1}
            E^{m}_{\beta, \alpha+\gamma}(at^{\beta}),
        \end{equation}
        and assuming, for example $\nu_1>\nu_0$, we find that
        \begin{align}
            &I^{\nu_1-\nu_0}\Pr \{ \mathcal{N}(t)=1\}\\
            \nonumber &= \sum_{m=0}^{\infty}(-1)^m \lambda^{m+1}\sum_{r=0}^{m}\binom{m}{r}
            I^{\nu_1-\nu_0}\left(t^{\nu_0(m-r)+\nu_1 r+\nu_0}E_{\nu_0+\nu_1,\nu_0(m-r)+\nu_1
            r+\nu_0+1}^{m+1}(-\lambda^2
            t^{\nu_0+\nu_1})\right)\\
            \nonumber &= \sum_{m=0}^{\infty}(-1)^m \lambda^{m+1}\sum_{r=0}^{m}\binom{m}{r}
            t^{\nu_0(m-r)+\nu_1 r+\nu_1}E_{\nu_0+\nu_1,\nu_0(m-r)+\nu_1
            r+\nu_1+1}^{m+1}(-\lambda^2
            t^{\nu_0+\nu_1})= p_1^{\nu_1}(t),
        \end{align}
        as expected.

        Moreover, we observe that, since
        $p_0^{\nu_0}(t)= \Pr \{ \mathcal{N}(t)=0 \}=E_{\nu_0,1}(-\lambda t^{\nu_0})$,
        we have
        \begin{equation}
            \sum_{k=1}^{\infty}p_k^{\nu_k}(t)= \sum_{k=1}^{\infty}\Pr \{ \mathcal{N}(t)=k\}
            = 1-E_{\nu_0,1}(-\lambda t^{\nu_0}).
        \end{equation}
        In view of \eqref{rel}, this implies that
        \begin{equation}
            \sum_{k=1}^{\infty}\Pr \{ \mathcal{N}(t)=k\}=  \sum_{k=1}^{\infty}p_k^{\nu_k}(t)=
            \sum_{k\colon \nu_k> \nu_0} I^{\nu_k-\nu_0} \Pr \{ \mathcal{N}(t)=k\}
            + \sum_{k\colon \nu_k<\nu_0} D^{\nu_0-\nu_k} \Pr \{ \mathcal{N}(t)=k\}.
        \end{equation}

        \bigskip

        The second process we construct here, denoted by $\widehat{N}(t)$, $t \ge 0$,
        is given by the following generalization of the Poisson process,
        whose univariate probabilities are given by
        \begin{equation}
            \label{p2}
            \Pr\{\widehat{N}(t)= j\}=\frac{\frac{(\lambda t)^j}{\Gamma(\nu_j j+1)}
            \frac{1}{E_{\nu_j,1}(\lambda t)}}{\sum_{j=0}^{+\infty} \frac{(\lambda t)^j}{\Gamma(\nu_j j+1)}
            \frac{1}{E_{\nu_j,1}(\lambda t)}}, \qquad j \geq 0,
        \end{equation}
        where $\lambda >0$, $0<\nu_j\leq 1$. We can treat it as a generalized Poisson process with state-dependent
        probabilities.
        Indeed, we notice that, if $\nu_j = 1$, for all $j$, we have
        \begin{align}
            \Pr\{\widehat{N}(t)= j\}&=\frac{\frac{(\lambda t)^j}{\Gamma(j+1)}
            \frac{1}{e^{\lambda t}}}{\sum_{j=0}^{+\infty} \frac{(\lambda t)^j}{\Gamma(j+1)}
            \frac{1}{e^{\lambda t}}}
            = \frac{(\lambda t)^j}{j!}e^{-\lambda t} = \Pr\{\mathrm N(t)= j\},
        \end{align}
        that is the state probability of the homogeneous Poisson process.

        A similar construction was adopted in \citet{Beghin}.
        We notice that an analogous generalization was used by
        \citet{Person} in quantum mechanics, in relation to
        Mittag-Leffler type coherent states.
        We now recall from \citet{Bala}
        that the distribution \eqref{p2}
        can be regarded as a weighted Poisson sum. Indeed we notice that
        \begin{equation}
            \frac{(\lambda t)^j}{\Gamma(\nu_j j+1)}\frac{1}{E_{\nu_j,1}(\lambda t)}
            = \frac{\frac{j!}{\Gamma(\nu_j j+1)}\Pr\{\mathrm N(t)=j\}}{\sum_{k=0}^{+\infty}
            \frac{k!}{\Gamma(\nu_j k+1)}\Pr\{\mathrm N(t)=k\}}.
        \end{equation}
        Hence we have
        \begin{equation}
            \Pr\{\widehat{N}(t)= j\}=
            \frac{\frac{\frac{j!}{\Gamma(\nu_j j+1)}
            \Pr\{\mathrm N(t)=j\}}{\sum_{k=0}^{+\infty}\frac{k!}{\Gamma(\nu_j k+1)}
            \Pr\{\mathrm N(t)=k\}}}{\sum_{j=0}^{+\infty}
            \frac{\frac{j!}{\Gamma(\nu_j j+1)}\Pr\{\mathrm N(t)=j\}}{\sum_{k=0}^{+\infty}
            \frac{k!}{\Gamma(\nu_j k+1)}\Pr\{\mathrm N(t)=k\}}}.
        \end{equation}

        The probability generating function of \eqref{p2} is given by
        \begin{align}
            G(u,t)&= \sum_{k=0}^{\infty}u^k \Pr\{\widehat{N}(t)= k\}\\
            &=  \frac{\sum_{k=0}^{+\infty}\frac{(\lambda u t)^k}{\Gamma(\nu_k k+1)}
            \frac{1}{E_{\nu_k,1}(\lambda t)}}{\sum_{k=0}^{+\infty} \frac{(\lambda t)^k}{\Gamma(\nu_k k+1)}
            \frac{1}{E_{\nu_k,1}(\lambda t)}}.\notag
        \end{align}
        In the case $\nu_j = \nu$, for all $j\geq 0$, we have
        \begin{align}
            G(u,t)= \frac{\sum_{k=0}^{+\infty}\frac{(\lambda u t)^k}{\Gamma(\nu k+1)}
            \frac{1}{E_{\nu,1}(\lambda t)}}{\sum_{k=0}^{+\infty} \frac{(\lambda t)^k}{\Gamma(\nu k+1)}
            \frac{1}{E_{\nu,1}(\lambda t)}} = \frac{E_{\nu,1}(u\lambda t)}{E_{\nu,1}(\lambda t)},
        \end{align}
    that coincides with the equation (4.4) of \citet{Beghin}.\\
        By means of the generating function we can also find the explicit form of the mean value of the
        distribution \eqref{p2}, i.e.
        \begin{align}
            \mathbb{E}\widehat{N}(t)&= \frac{\lambda t\sum_{k=0}^{+\infty}
            \frac{k(\lambda t)^{k-1}}{\Gamma(\nu_k k+1)}\frac{1}{E_{\nu_k,1}
            (\lambda t)}}{\sum_{k=0}^{+\infty} \frac{(\lambda t)^k}{\Gamma(\nu_k k+1)}\frac{1}{E_{\nu_k,1}(\lambda t)}}
            = \frac{\lambda t\sum_{k=0}^{+\infty}\frac{(\lambda t)^{k}}{\nu_{k+1}\Gamma(\nu_{k+1}
            k+\nu_{k+1})}\frac{1}{E_{\nu_{k+1},1}(\lambda t)}}{\sum_{k=0}^{+\infty}
            \frac{(\lambda t)^k}{\Gamma(\nu_k k+1)}\frac{1}{E_{\nu_k,1}(\lambda t)}},
        \end{align}
        such that, when $\nu_k = \nu$ for all $k$, we recover the case considered in \citet{Beghin}
        and in \citet{Macci}, i.e.\
        \begin{align}
            \mathbb{E}\widehat{N}_{\nu}(t)=\frac{\lambda t E_{\nu, \nu}(\lambda t)}{\nu E_{\nu,1}(\lambda t)}.
        \end{align}

        We now consider a sequence of a random number of non-negative i.i.d.\ random variables with distribution
        $F(\beta)=\Pr(X_i\leq \beta)$, $i\geq 1$ and represented
        by $\widehat{N}(t)$. The distribution of the maximum and minimum of this sequence is given by
        \begin{align}
            \Pr\{\max\left(X_1,\cdots, X_{\widehat{N}(t)}\right)<\beta\}
            &= \frac{\sum_{k=0}^{+\infty}\frac{(\lambda F(\beta) t)^k}{\Gamma(\nu_k k+1)}
            \frac{1}{E_{\nu_k,1}(\lambda t)}}{\sum_{k=0}^{+\infty} \frac{(\lambda t)^k}{\Gamma(\nu_k k+1)}
            \frac{1}{E_{\nu_k,1}(\lambda t)}},\\
            \Pr\{\min\left(X_1,\cdots, X_{\widehat{N}(t)}\right)>\beta\}
            &= \frac{\sum_{k=0}^{+\infty}\frac{(\lambda[1- F(\beta)] t)^k}{\Gamma(\nu_k k+1)}
            \frac{1}{E_{\nu_k,1}(\lambda t)}}{\sum_{k=0}^{+\infty}
            \frac{(\lambda t)^k}{\Gamma(\nu_k k+1)}\frac{1}{E_{\nu_k,1}(\lambda t)}}.
        \end{align}
        In the case $\nu = \nu_k = 1$, for all $k$, we recover the distribution of the maximum and minimum
        of the homogeneous Poisson process.

    \section{State dependent fractional pure birth processes}

        In this section we consider a different point process which can be generalized in a state-dependent
        sense as we have done for the fractional Poisson process.
        We thus analyze a state-dependent fractional pure birth process (see \citet{fede} for the
        fractional case with constant order), where the probabilities
        are governed by the following equations

        \begin{align}
            \label{bir1}
            \begin{cases}
                \frac{\mathrm{d}^{\nu_k}}{\mathrm{d}t^{\nu_k}} p_k^{\nu_k}(t)
                = -\lambda_k p_k^{\nu_k}(t) +\lambda_{k-1} p_{k-1}^{\nu_{k-1}}(t), & k \ge 1,
                \: t > 0, \: \nu_k \in (0,1], \\
                p_k^{\nu_k}(0) =
                \begin{cases}
                    1, \quad k = 1, \\
                    0, \quad k \geq 2.
                \end{cases}
            \end{cases}
        \end{align}

        As in the Section \ref{due} the Laplace transform of the solution to \eqref{bir1} can be found rather
        easily. This is done in the following proposition.
        \begin{prop}
            The Laplace transform of the solution to the state-dependent
            fractional pure-birth process \eqref{bir1} reads
            \begin{align}
                \tilde{p}_k^{\nu_k} (s) = \int_0^{+\infty}e^{-st}p_k^{\nu_k}(t) \mathrm dt
                = \left(\prod_{j=1}^{k-1}\lambda_j\right) \frac{s^{\nu_1-1}}{\prod_{j=1}^k
                (s^{\nu_j}+\lambda_j)},
            \end{align}
            where the fractional derivative appearing in \eqref{bir1} is
            in the sense of Dzhrbashyan--Caputo.
        \end{prop}

            \begin{proof}
                We can solve equation \eqref{bir1} by means of an iterative
                procedure, as follows.
                The equation related to $k=1$
                \begin{align}
                    \label{bir0}
                    \begin{cases}
                        \frac{\mathrm{d}^{\nu_1}}{\mathrm{d}t^{\nu_1}} p_1^{\nu_1}(t)
                        = -\lambda_1 p_1^{\nu_1}(t), & \: t > 0, \: \nu_1 \in (0,1], \\
                        p_1^{\nu_1}(0) = 1,
                    \end{cases}
                \end{align}
                has solution
                $p_1^{\nu_1}(t)= E_{\nu_1,1}(-\lambda t^{\nu_1})$.
                For $k=2$, the equation

                \begin{align}
                    \label{bir2}
                    \begin{cases}
                        \frac{\mathrm{d}^{\nu_2}}{\mathrm{d}t^{\nu_2}} p_2^{\nu_2}(t)
                        = -\lambda_2 p_2^{\nu_2}(t) + \lambda_1 p_1^{\nu_1}(t),
                        & \: t > 0, \: \nu_2 \in (0,1], \\
                        p_2^{\nu_2}(0) = 0,
                    \end{cases}
                \end{align}
                has solution with Laplace transform
                \begin{equation}
                    \tilde{p}_2^{\nu_2} (s)=\int_0^{+\infty}e^{-st}p_2^{\nu_2}(t) \mathrm dt =
                    \frac{\lambda_1 s^{\nu_1-1}}{\lambda_1+s^{\nu_1}}\frac{1}{\lambda_2+s^{\nu_2}}.
                \end{equation}
                whose inverse is given by (see \eqref{conto})
                \begin{equation}
                    p_2^{\nu_2} (t)=\sum_{m=0}^{\infty}(-1)^m \sum_{r=0}^{m}\binom{m}{r}\lambda_1^{r+1}\lambda_2^{m-r}
                    t^{\nu_2(m-r)+\nu_1 r+\nu_2}E_{\nu_1+\nu_2,\nu_2(m-r)+\nu_1
                    r+\nu_2+1}^{m+1}(-\lambda_1 \lambda_2
                    t^{\nu_1+\nu_2}).
                \end{equation}
                By iterating this procedure, we arrive immediately at
                \begin{align}
                    \label{bir3}
                    \tilde{p}_k^{\nu_k} (s) = \int_0^{+\infty}e^{-st}p_k^{\nu_k}(t) \mathrm dt
                    =\left(\prod_{j=1}^{k-1}\lambda_j\right) \frac{s^{\nu_1-1}}{\prod_{j=1}^k
                    (s^{\nu_j}+\lambda_j)},
                \end{align}
                as claimed.
            \end{proof}

        By recalling \eqref{ricos}, we obtain the explicit expression of the state probabilities $p_k^{\nu_k}(t)$,
        $k \ge 1$, $t \ge 0$, as
        \begin{equation}
            p_k^{\nu_k} (t)= E_{\nu_1,1}(-\lambda_1 t^{\nu_1})
            \bigast_{j=1}^k \lambda_j t^{\nu_j-1}E_{\nu_j,\nu_j}(-\lambda_j
            t^{\nu_j}),
        \end{equation}
    where the convolution is in the sense of equation \eqref{praba}.

        We now consider the state dependent linear birth process, denoted
        by $N_{lin}(t)$, $t \ge 0$. This means that we take $\lambda_k =
        \lambda k$ in \eqref{bir1}. We have the following
        \begin{te}
            Let us consider the state dependent linear birth process $N_{lin}(t)$, $t \ge
            0$, governed by
            \begin{equation}
                \label{birl}
                \begin{cases}
                    \frac{\mathrm{d}^{\nu_k}}{\mathrm{d}t^{\nu_k}} p_k^{\nu_k}(t)
                    = -\lambda k p_k^{\nu_k}(t) +\lambda (k-1) p_{k-1}^{\nu_{k-1}}(t), \: k \ge 1, \: t > 0, \: \nu_k \in
                    (0,1],\\
                    p_k^{\nu_k}(0) =
                    \begin{cases}
                        1, \quad k = 1, \\
                        0, \quad k \geq 2,
                    \end{cases}
                \end{cases}
            \end{equation}
            then the following relation holds
            \begin{align}
                \sum_{k=1}^{\infty}k^m\frac{\mathrm{d}^{\nu_k}}{\mathrm{d}t^{\nu_k}} p_k^{\nu_k}(t)
                &= \lambda
                \sum_{j=1}^{m-1}\binom{m}{j}\mathbb{E}N_{lin}^{m-j+1}.
            \end{align}
            \end{te}

            \begin{proof}
                In order to find explicit relations for the moments of the
                distribution $N_{lin}(t)$, we multiply both sides of equation \eqref{birl} by $k^m$
                and sum over all the states,
                obtaining
                \begin{align}
                    \sum_{k=1}^{\infty}k^m\frac{\mathrm{d}^{\nu_k}}{\mathrm{d}t^{\nu_k}} p_k^{\nu_k}(t)
                    &= -\lambda \sum_{k=1}^{\infty}k^{m+1}p_k^{\nu_k}(t) +\lambda\sum_{k=1}^{\infty}
                    k^m(k-1) p_{k-1}^{\nu_{k-1}}(t)\\
                    \nonumber &= -\lambda \sum_{k=1}^{\infty} k^{m+1} p_k^{\nu_k}(t)+\lambda \sum_{k=1}^{\infty}
                    k(k+1)^m p_k^{\nu_k}(t)\\
                    \nonumber &= -\lambda \sum_{k=1}^{\infty} k^{m+1} p_k^{\nu_k}(t)
                    +\lambda \sum_{k=1}^{\infty} \sum_{j=0}^{m}\binom{m}{j}k^{m-j+1}
                    p_k^{\nu_k}(t)\\
                    \nonumber & = \lambda \sum_{j=1}^{m}\binom{m}{j}\sum_{k=1}^{\infty}k^{m-j+1}
                    p_k^{\nu_k}(t)= \lambda
                    \sum_{j=1}^{m}\binom{m}{j}\mathbb{E}N_{lin}^{m-j+1},
                \end{align}
            \end{proof}

        \begin{os}
            We can consider in a explicit way the relations involving first and second moments.
            For example, if we multiply \eqref{birl} for $k$ and sum over all the states,
            we obtain that
            \begin{align}
                \sum_{k=1}^{\infty}k\frac{\mathrm{d}^{\nu_k}}{\mathrm{d}t^{\nu_k}} p_k^{\nu_k}(t)
                &= -\lambda \sum_{k=1}^{\infty}k^2 p_k^{\nu_k}(t) +\lambda \sum_{k=1}^{\infty} k(k-1)
                p_{k-1}^{\nu_{k-1}} (t)\\
                \nonumber &= \lambda \sum_{k=1}^{\infty} k p_k^{\nu_k}(t)= \lambda
                \mathbb{E} N_{lin}(t).
            \end{align}
            In the same way, for the second moment, we multiply \eqref{birl} for
            $k^2$, obtaining
            \begin{align}
                \sum_{k=1}^{\infty}k^2\frac{\mathrm{d}^{\nu_k}}{\mathrm{d}t^{\nu_k}} p_k^{\nu_k}(t)
                &= -\lambda \sum_{k=1}^{\infty}k^3 p_k^{\nu_k}(t) +\lambda\sum_{k=1}^{\infty}
                k^2(k-1) p_{k-1}^{\nu_{k-1}}(t)\\
                \nonumber &= \lambda \sum_{k=1}^{\infty} k p_k^{\nu_k}(t)+2\lambda \sum_{k=1}^{\infty}
                k^2 p_k^{\nu_k} (t)\\
                \nonumber &= \lambda \mathbb{E}
                N_{lin}(t)+2\lambda \mathbb{E}(N_{lin})^2(t).
            \end{align}
        \end{os}

        \subsection*{Aknowledgements}
            The authors thank the referee for his (her) appreciation
            of our work and for his (her) useful remarks.\\
            F. Polito has been supported by project AMALFI (Universit\`{a} di Torino/Compagnia di San Paolo).

\end{document}